\documentclass[12pt]{amsart}
\usepackage{amscd}
\usepackage{graphicx}
\def\frk{\frak}               

\def\Phi{{\frk n}}
\def\Phi{{\frk N}}
\def\opn#1#2{\def#1{\operatorname{#2}}} 
\opn\chara{char} \opn\length{\ell} \opn\pd{pd} \opn\rk{rk}
\opn\projdim{proj\,dim} \opn\injdim{inj\,dim} \opn\rank{rank}
\opn\depth{depth} \opn\grade{grade} \opn\height{height}
\opn\embdim{emb\,dim} \opn\codim{codim}

\opn\Tr{Tr} \opn\bigrank{big\,rank}
\opn\superheight{superheight}\opn\lcm{lcm}
\opn\trdeg{tr\,deg}
\opn\reg{reg} \opn\lreg{lreg} \opn\ini{in} \opn\lpd{lpd}
\opn\size{size}
\opn\div{div} \opn\Div{Div} \opn\cl{cl} \opn\Cl{Cl}
\opn\Spec{Spec} \opn\Supp{Supp} \opn\supp{supp} \opn\Sing{Sing}
\opn\Ass{Ass} \opn\Min{Min} \opn\Shad{Shadow}
\opn\Ann{Ann} \opn\Rad{Rad} \opn\Soc{Soc}
\opn\Im{Im} \opn\Ker{Ker} \opn\Coker{Coker} \opn\Am{Am}
\opn\Hom{Hom} \opn\Tor{Tor} \opn\Ext{Ext} \opn\End{End}
\opn\Aut{Aut} \opn\id{id}

\opn\nat{nat}
\opn\pff{pf}
\opn\Pf{Pf} \opn\GL{GL} \opn\SL{SL} \opn\mod{mod} \opn\ord{ord}
\opn\Gin{Gin} \opn\Hilb{Hilb}
\opn\aff{aff} \opn\con{conv} \opn\relint{relint} \opn\st{st}
\opn\lk{lk} \opn\cn{cn} \opn\core{core} \opn\vol{vol}
\opn\link{link} \opn\star{star}
\opn\gr{gr}

\def\pot#1#2{#1[\kern-0.28ex[#2]\kern-0.28ex]}

\opn\dirlim{\underrightarrow{\lim}}
\opn\inivlim{\underleftarrow{\lim}}
\def\Implies{\ifmmode\Longrightarrow \else
        \unskip${}\Longrightarrow{}$\ignorespaces\fi}
\def\implies{\ifmmode\Rightarrow \else
        \unskip${}\Rightarrow{}$\ignorespaces\fi}
\def\iff{\ifmmode\Longleftrightarrow \else
        \unskip${}\Longleftrightarrow{}$\ignorespaces\fi}

\let\:=\colon
\newtheorem{Theorem}{Theorem}[section]

\newtheorem{Proposition}[Theorem]{Proposition}
\newtheorem{Remark}[Theorem]{Remark}

\newtheorem{Example}[Theorem]{Example}

\newtheorem{Definition}[Theorem]{Definition}

\let\epsilon\varepsilon
\let\phi=\varphi
\let\kappa=\varkappa
\def\qed{\ifhmode\textqed\fi
      \ifmmode\ifinner\quad\qedsymbol\else\dispqed\fi\fi}
\def\textqed{\unskip\nobreak\penalty50
       \hskip2em\hbox{}\nobreak\hfil\qedsymbol
       \parfillskip=0pt \finalhyphendemerits=0}
\def\dispqed{\rlap{\qquad\qedsymbol}}

\opn\dis{dis}
\def\pnt{{\raise0.5mm\hbox{\large\bf.}}}

\opn\Lex{Lex}

\begin{document}

%


\title{Quasi $f$-Ideals  $^{*}$  }

\author{ Hasan Mahmood$^{1}$, Fazal Ur Rehman$^{1}$, Thai Thanh Nguyen $^{3}$, Muhammad Ahsan Binyamin $^{2}$
}
\thanks{\noindent $^{*}$ The first and the last authors are supported by the Higher Education Commission of Pakistan for this research (Grant no. 7515).\\
\noindent $^{1}$Government College University Lahore, Pakistan. $^{2}$Government College University Faisalabad, Pakistan.
$^{3}$ Tulane University, USA; Hue University, College of Education, Vietnam.\\
{\em E-mails }: hasanmahmood@gcu.edu.pk,
fazalqau@gmail.com, ahsanbanyamin@gmail.com,  tnguyen11@tulane.edu }
\maketitle
\begin{abstract}
  The notion of $f$-ideals is recent and has been studied in the papers \cite{deg2} \cite{degd}, \cite{Adam}, \cite{tswu1}, \cite{tswu2}, \cite{tswuPubl}, \cite{fgraph}, \cite{fsimp}  and \cite{fnote}. In this paper, we have generalized the idea of $f$-ideals to quasi $f$-ideals. This extended class of ideals is much bigger than the class of all $f$-ideals. Apart from giving various characterizations of quasi $f$-ideals of degree 2,  we have determined all the minimal primes ideals of these ideals. Moreover, construction of quasi $f$-ideals of degree 2 has been described; the formula for computing Hilbert function and Hilbert series of the polynomial ring modulo quasi $f$-ideal has been provided.
 \vskip 0.4 true cm
\noindent
\end{abstract}

\section{Introduction}

The notion of $f$-vector has a fundamental importance in algebraic, topological, and combinatorial study of simplicial complexes and polytopes. It has been studied since the time of Leonhard Euler, for example, see \cite{CMRBH}, \cite{TP} and \cite{HHBINOM}. The $f$-ideals, which were first introduced in \cite{deg2}, involve the idea of $f$-vectors of two important simplicial complexes. More precisely, a square-free monomial ideal $I$ of the polynomial ring $R=k[x_{1},x_{2},...,x_{n}]$ (where $k$ is a field) is an $f$-ideal if and only if the $f$-vector of the facet complex  of $I$ coincides with the $f$-vector of its Stanley-Reisner complex. These ideal were first studied in \cite{deg2}, where the authors gave a characterization of $f$-ideals of degree 2.  This characterization was somehow algebraic in its nature as it required the square-free monomial ideal of $R$ to be unmixed of height $n-2$. The definition of $f$-ideal is a blend of combinatorics, algebra, and topology. In order to characterize $f$-ideals for any degree, it was needed to see it through combinatorial and topological aspects too. The characterization of $f$-ideals for homogeneous unmixed square-free monomial ideals of any degree $d$  was given in \cite{degd}, in which combinatorial aspects were also considered. Later on, the notion of $f$-graphs (in \cite{fgraph}) and $f$-simplicial complex (in \cite{fsimp}) was introduced. These notions have been studied for it various properties in the papers \cite{deg2}, \cite{degd}, \cite{Adam}, \cite{tswu1}, \cite{tswu2}, \cite{tswuPubl}, \cite{fgraph}, \cite{fsimp} and \cite{fnote}.\\

In Computational Algebraic Geometry and Commutative Algebra, the notion of Hilbert polynomial and  Hilbert series are very useful and important invariants of any finitely generated standard graded algebra over some field. These encode much useful information and are the easiest way for the computation of degree and dimension of an algebraic variety defined through explicit polynomials. The Theorem 6.7.2 and the proposition 6.7.3 of \cite{v} tell that Hilbert function, and Hilbert series of the Stanley-Reisner ring can be computed through the $f$-vector of non-face complex. It is clear that finding $f$-vector of the facet complex of some square-free monomial ideal is much simpler than computing $f$-vector of its non-face complex. One importance of studying $f$-ideals lies in the fact that Hilbert series of polynomial ring modulo any $f$-ideal can be computed directly by looking at the $f$-vector of its facet complex only. However, the class of $f$-ideals is not that large; in addition,  $f$-ideals exist in $R=k[x_{1},x_{2},...,x_{n}]$ only for special $n$'s. So far no criterion exists in the facet ideal theory which helps us in computing Hilbert function and Hilbert series of the polynomial ring modulo any square-free monomial ideal by using the $f$-vector of the facet complex of $I$.\\

 \indent The motivation of writing this article is to extend the class of those ideals for which the Hilbert function and the Hilbert series of $R/I$ can be computed through the $f$-vector of their facet complex. The idea is to read off the $f$-vector of the non-face complex of $I$ with the help of $f$-vector of its facet complex. For example, consider the ideal $I= \langle x_1x_2, x_3x_4, x_1x_3x_5, x_2x_4x_5 \rangle$ in the polynomial ring  $R=k[x_{1},x_{2},x_{3},x_{4},x_{5}]$. This ideal $I$ is $f$-ideal (see \cite{tswu2}); the common $f$-vector of the facet complex and the non-face complex of $I$ is $(5,8,2)$. Then, by \cite[Theorem 6.7.2]{v}, the Hilbert series of $R/I$ is equal to $\frac{1z^0}{(1-z)^0}+\frac{5z^1}{(1-z)^1}+\frac{8z^2}{(1-z)^2}+\frac{2z^3}{(1-z)^3}$. However, in the same ring, the ideal $J= \langle x_1x_2x_4,x_1x_2x_5,x_1x_4x_5,x_2x_3x_5,x_3x_4x_5\rangle$ is not $f$-ideal because the $f$-vector of its facet complex is $(5,9,10)$ and the $f$-vector of its Stanley-Reisner complex is $(5,10,10)$. Although $J$ is not $f$-ideal, yet we can still express the Hilbert series of $R/J$ in terms of the $f$-vector of its facet complex in the following manner: $\frac{1z^0}{(1-z)^0}+\frac{(5+0)z^1}{(1-z)^1}+\frac{(9+1)z^2}{(1-z)^2}+\frac{(10+0)z^3}{(1-z)^3}$ , keeping in mind that $(0,1,0)$ is the difference vector. The key point is to control this difference vector.  Once we are able to control the difference of $f$-vectors of these two complexes, we can obviously achieve the target of expressing Hilbert series of polynomial ring modulo the ideal. It is natural to name the ideal $J$ as  'quasi $f$-ideal of type $(0,1,0)$'. Before giving a systematic definition of quasi $f$-ideals, we would like to remark that every $f$-ideal will turn out to be quasi $f$-ideal of type $0$ vector; moreover, unlike $f$-ideals, these ideals can be found in any polynomial ring in any number of variables.\\

 \indent This paper is organized as follows: section 2 is devoted to recalling some basic definitions to make this paper self-explanatory. In the third section, we give a systematic definition of quasi $f$-ideal supported with some examples. The section 4  supplies two different characterizations of equigenerated quasi $f$-ideals of degree $2$, given in Theorem 4.1 and Theorem 4.3; the proposition 4.4 answers the question that which ordered pairs of $\mathbb{Z}^2 $ can be realized as the type of some quasi $f$-ideals. The Theorem 4.5 gives all the minimal prime ideals of any homogeneous quasi $f$-ideal of degree $2$. Then a construction of these ideals has been given in the proposition 4.7; a formulation of Hilbert function and Hilbert series for these ideals is given in Theorem 4.9 and Theorem 4.10 respectively.

\section{Basic Set Up}

This section entails basic definitions and concepts which make this
article self-contained. Throughout this paper, the character $k$ represents a field, and $R$ is a polynomial ring over $k$ in $n$ variables $x_1,x_2,\dots,x_n$. We start with the following definition of a
simplicial complex.

\begin{Definition}{\em
Let $V=\{v_{1},v_{2},...,v_{n}\}$ be a vertex set and $\Delta$ be a
subset of $P(V)$. We say $\Delta$ a simplicial complex on $V$ if,
$(i)$ $\{v_{i}\}\in\Delta$ for all $i\in\{1,2,\ldots,n\}$, and,
$(ii)$ subsets of every element of $\Delta$ belong to $\Delta$.}

\end{Definition}
The
members of $\Delta$ are known as faces; the dimension of a face $F$
is one less than the cardinality of $F$. The maximal faces under
inclusion are known as facets. It is clear that a simplicial complex can be determined by its facets. If $F_1,F_2,\ldots,F_r$ are the facets of $\Delta$, we write $\Delta=\langle F_{1},F_{2},...,F_{r}\rangle$ to say that $\Delta$ is generated by these $F_i's$. The dimension of a simplicial complex $\Delta$ is
defined as follows: $$\dim(\Delta)=\max\{\dim(F)|F
\text{ is facet in }
\Delta\}$$
\begin{Remark}{\em  A simplicial complex whose facets can have at most dimension $1$ is actually a simple graph, where the $1$-dimensional facets are termed as edges, and the $0$-dimensional facets are isolated vertices. A graph is usually denoted by $G$ with $E(G)$ being the set of its edges. }

\end{Remark}
%

We need to recall few more concepts from the literature before giving the definition of (quasi) $f$-ideals.

\begin{Definition}
{\em The $f$-vector of a $d$-dimensional simplicial complex $
\Delta$ is an element $(f_{0},f_{1},...,f_{d})\in {\mathbb{Z}}^{d+1}$, where
$f_{i}=\left\vert \left\{ F\in \Delta :\dim (F)=i\right\}
\right\vert $ for all $i\in\{0,1,2,...,d\}$}. The $f$-vector of $\Delta$ is denoted by $f(\Delta )$.
\end{Definition}

\begin{Definition}({\bf{facet complex and non-face complex}})
{\em Consider a square-free monomial ideal $I$ of $R=k[x_1,x_2,\ldots,x_n]$ with $G(I)=\{m_{1},m_{2},...,m_{r}\}$ as its unique minimal monomial system of generators. Then $F_{1},F_{2},..., F_{r}$ are the facets of the facet complex of $I$ on the vertices $v_1,v_2,\ldots,v_n$, where $F_{i}=\{v_{j}:x_{j}$ divide
$m_{i}\}$, where $i\in \{1,2,...,r\}$. The facet complex of $I$ is denoted by $\delta _{\mathcal{F}}(I)$.
And, the non-face complex of $I$  is a simplicial complex on $V=\{v_1,v_2,\ldots v_n\}$ such that a subset
$\{v_{i_{1}},v_{i_{2}},...,v_{i_{k}}\}$ of $V$ is a face of this non-face complex if and only if the corresponding monomial
$x_{i_{1}}x_{i_{2}}\cdots x_{i_{k}}$ does not belong to $I$. The non-face complex of $I$ is also known as the Stanley-Reisner complex of $I$, and we denote it by $\delta
_{\mathcal{N}}(I)$. }
\end{Definition}

\begin{Definition} ({\bf{facet ideal and non-face ideal}})
{\em Consider a simplicial $\Delta$ on the vertex $V=\{v_1,v_2,\ldots,v_n\}$ which is generated by the facets $F_1,F_2,\ldots,F_r$. The
facet ideal of $\Delta $ is
square-free monomial ideal of $R=k[x_1,x_2,\ldots,x_n]$ which is minimally generated by
the monomials $m_{1},m_{2},...,m_{r}$ such that
$m_{i}=\prod\limits_{v_{j}\in F_{i}}x_{j}$, where $i\in
\{1,2,...,r\}$. The facet ideal of $\Delta$ is denoted by $I_{\mathcal{F}}(\Delta )$. And, the non-face ideal of $\Delta$ is another square-free
monomial ideal of $R$
such that any monomial $x_{i_{1}}x_{i_{2}}...x_{i_{k}}$ is in the non-face ideal if and only if  the corresponding subset
 $\{v_{i_{1}},v_{i_{2}},...,v_{i_{k}}\}$ of $V$ does not belong to the complex $\Delta$. The non-face ideal of $\Delta$ is also known as its Stanley-Reisner ideal, and it is written as $I_{\mathcal{N}}(\Delta )$. }
\end{Definition}

\begin{Definition}
{\em A square-free monomial ideal $I$ of the polynomial ring $R$ is said to be an $f$-ideal if and only if the $f$-vector of the facet complex  of $I$ coincides with the $f$-vector of its Stanley-Reisner complex, i.e. $f(\delta _{\mathcal{F}}(I))=f(\delta _{\mathcal{N}}(I))$. A simplicial complex is said to be an $f$-simplicial complex if its facet ideal is an $f$-ideal. A $1$-dimensional $f$-simplicial complex is termed as $f$-graph for obvious reason. }
\end{Definition}

Let us place the definition and some examples of the central notion of this paper, i.e., quasi $f$-ideal, in the separate section. However, before moving to the next section, we would rather recall the definition of perfect sets of $R$. These sets are used in characterizing $f$-ideals as given in \cite[Theorem 2.3]{tswuPubl}.\\

Let $Sm(R)$ denote the set of all square-free monomials in $R$; let $Sm(R)_d$ be the set of all square-free monomials of degree $d$ in $Sm(R)$.  For a subset $T\subseteq Sm(R)$, consider the upper shadow, $\sqcup(T)$, and the lower shadow of $T$, $\sqcap(T)$, as given below:
$$\sqcup(T)=\{gx_{i} \ | \ g\in T, {x_{i}} \text{  does not divide  } g, 1\leq
i\leq n\}$$ $$\sqcap(T)=\{h \ | \ h=g/x_{i} \text{  for some  } g\in T  \text{  and some  } x_{i} \text{
with  } x_{i}|g\}$$

If, in particular, $T$ sits in $Sm(R)_d$, then $\sqcup(T)\subset Sm(R)_{d+1}$ and $\sqcap(T)\subset Sm(R)_{d-1}$. The set $T$ is then called upper perfect if $\sqcup(T)= Sm(R)_{d+1}$, and it is said to be lower perfect if its lower shadow is full, i.e., $\sqcap(T)= Sm(R)_{d-1}$. The set $T$ is called a perfect set if and only if it is lower perfect as well as upper perfect.
In general, perfect sets can have different cardinalities; for example, every subset of $Sm(R)_{d}$ containing a perfect set is again a perfect set. The smallest number among the cardinalities of perfect sets of degree ${d}$ is called the $(n,d)^{th}$ perfect number, and is denoted by
 $N{(n,d)}$. By \cite[Lemma 3.3]{tswuPubl}, for a positive $t$ and $n\geq4$, we have the following equations:

  $$ N(n,2)= \left\{
                                                                                    \begin{array}{ll}
                                                                                      t^2-t, & \hbox{ when $n=2t$;} \\
                                                                                      t^2, & \hbox{ when $n = 2t+1$.}
                                                                                    \end{array}
                                                                                  \right.
$$

 \section{Quasi $f$-Ideals: Definition and Examples}
 In this section, we have introduced the notion of quasi $f$-ideals, quasi $f$-graphs and quasi $f$-simplicial complexes. Some examples are also presented.

 \begin{Definition}{\em
Let $(a_1,a_2,\ldots,a_s)\in \mathbb{Z}^s $. A square-free monomial ideal $I$ in the polynomial ring
$R=k[x_{1},x_{2,}...,x_{n}]$  is said to be a quasi $f$-ideal
of type $(a_1,a_2,\ldots,a_s)$ if and only if  $f(\delta _{\mathcal{N}}(I))- f(\delta _{\mathcal{F}}(I))=\left(
a_1,a_2,\ldots,a_s\right) $.}
\end{Definition}

If $I$ is quasi $f$-ideal of type $(a_1,a_2,\ldots,a_s)$ in the ring $R=k[x_{1},x_{2,}...,x_{n}]$, then the definition requires that both the complexes of $I$, the facet complex and the Stanley-Reisner complex, should be $s$-dimensional. It means that $s\leq n$; in fact,  $s+1 = \max\{\deg(u) : u\in G(I)\}$, where G(I) is the set of minimal generators of $I$. Moreover, as a consequence of Kruskal-Katona theorem, we can say that not every $s$-tuple of integers can be realized as type of some quasi $f$-ideal. However, every $f$-ideal is a quasi $f$-ideal whose type is a zero vector, and obviously any quasi $f$-ideal with type some non-zero vector can not be $f$-ideal. But we would like to mention that the class of quasi $f$-ideals is much more bigger than the class of $f$-ideals; moreover, unlike $f$-ideals, examples of quasi $f$-ideals can be found in $R=k[x_{1},x_{2,}...,x_{n}]$, for any $n$.


 Let us now consider some examples of quasi $f$-ideals of some types below.

 \begin{Example}{\em
Every $f$-ideal is quasi $f$-ideal of type {\bf{\underline{0}}}. We would like the readers to see \cite{deg2},\cite{degd}, \cite{tswu2},\cite{tswuPubl} and \cite{fgraph} to know more about $f$-ideals and $f$-graphs.}
\end{Example}

\begin{Example}
{\em The ideal $J=\langle x_1x_2x_4,x_1x_2x_5,x_1x_4x_5,x_2x_3x_5,x_3x_4x_5 \rangle$ of the polynomial ring in $5$ variables, which has been discussed in the introduction of this paper, is quasi $f$-ideal of type $(0,1,0)$.}
\end{Example}

\begin{Example}{\em
Consider the monomial ideal $I=\langle x_1x_2x_6, x_1x_2x_7,x_1x_3x_4,x_1x_3x_5,$
$x_1x_3x_6,x_1x_3x_7,x_1x_4x_5,x_1x_4x_6,x_1x_5x_7,x_1x_6x_7,x_2x_4x_5,x_2x_4x_7,x_2x_6x_7,x_3x_4x_6,x_3x_5x_7,$
$x_2x_5x_6,x_5x_6x_7 \rangle $ of degree $3$ in the ring $R=k[x_{1},x_{2},x_{3},x_{4},x_{5},x_{6},x_{7}]$. The facet complex and the non-face complex of $I$ are
$$\delta _{\mathcal{F}}(I)=\langle\{v_{1},v_{2},v_{6}\}, \{v_{1},v_{2},v_{7}\}, \{v_{1},v_{3},v_{4}\}, \{v_{1},v_{3},v_{5}\},\{v_{1},v_{3},v_{6}\}, \{v_{1},v_{3},v_{7}\}$$
$$\{v_{1},v_{4},v_{5}\}, \{v_{1},v_{4},v_{6}\}, \{v_{1},v_{5},v_{7}\}, \{v_{1},v_{6},v_{7}\}, \{v_{2},v_{4},v_{5}\}, \{v_{2},v_{4},v_{7}\}$$ $$\{v_{2},v_{6},v_{7}\}, \{v_{3},v_{4},v_{6}\}, \{v_{3},v_{5},v_{7}\}, \{v_{2},v_{5},v_{6}\}, \{v_{5},v_{6},v_{7}\}\rangle$$ and
$$\delta _{\mathcal{N}}(I)=\langle\{v_{1},v_{2},v_{3}\}, \{v_{1},v_{2},v_{4}\}, \{v_{2},v_{3},v_{4}\}, \{v_{2},v_{3},v_{6}\},\{v_{1},v_{5},v_{6}\}, \{v_{3},v_{5},v_{6}\}$$
$$\{v_{3},v_{4},v_{5}\}, \{v_{2},v_{5},v_{7}\}, \{v_{4},v_{5},v_{7}\}, \{v_{4},v_{6},v_{7}\}, \{v_{3},v_{6},v_{7}\}, \{v_{4},v_{5},v_{6}\}, \{v_{2},v_{4},v_{6}\}$$ $$\{v_{3},v_{4},v_{7}\}, \{v_{1},v_{4},v_{7}\}, \{v_{2},v_{3},v_{7}\}, \{v_{2},v_{3},v_{5}\}, \{v_{1},v_{2},v_{5}\}\rangle$$
this implies that $f(\delta _{\mathcal{F}}(I))=(7,20,17)$ and $f(\delta _{\mathcal{N}}(I)))=(7,21,18)$, it means that $I$ is a quasi $f$-ideal of the type $(0,1,1)$.}
\end{Example}

\section{Quasi $f$-Ideals of Degree 2}
 Now we want to characterize those quasi $f$-ideals in the ring $R=k[x_{1},x_{2},...,x_{n}]$ whose minimal generating set is a subset of the set of all square-free monomial of degree 2 in $R$. We may call such ideals as equigenerated (or pure) square-free monomial ideals of degree 2. Obviously, the type of such quasi $f$-ideals will be ordered pair $(a,b)\in\mathbb{Z}^2$. However, since we are to consider only those ideals for which the facet complex and the non-face complex have the same vertex set $V=\{v_1, v_2,v_3,\ldots,v_n\}$, so we are bound to consider only those ideals whose support is full, i.e. an ideal $I$ of $R=k[x_{1},x_{2},...,x_{n}]$ for which $\bigcup_{u\in G(I)}supp(u)=\{x_1,x_2,\ldots,x_n\}$. This means that if $I$ is quasi $f$-ideal of type $(a,b)$, then $a$ must be zero in the ordered pair $(a,b)$. Thus any quasi $f$-ideal of degree 2 must be of the type $(0,b)$. The following theorem characterizes all such ideals.

 \begin{Theorem}{\em
Let $I$ be an equigenerated square-free monomial ideal of $R=k[x_{1},x_{2},...,x_{n}]$ of degree $2$, and let  $G(I)=\{u_1,u_2,\ldots,u_r\}$ be the unique minimal generating set of $I$. Then $I$ is quasi $f$-ideal of type $(0,b)\in \mathbb{Z}^2$ if and only if the following conditions hold true:
\begin{enumerate}
    \item(i) $ht(I)=n-2$;
    \item ${n\choose 2}\equiv  \begin{cases}
                           0\  (mod\  2) & \text{if b is even} \\
                           1\  (mod\  2) & \text{if b is odd}
                         \end{cases}
$;
    \item $|G(I)|=\frac{1}{2}({n\choose 2}-b)$.
\end{enumerate}

}
\end{Theorem}

\begin{proof}
First suppose that $I$ is quasi $f$-ideal of type $(0,b)$. This means that $dim (\delta _{\mathcal{F}}(I))$ $=dim (\delta _{\mathcal{N}}(I))$ and $f(\delta _{\mathcal{N}}(I))- f(\delta _{\mathcal{F}}(I))=\left(
0,b\right) $. But as $I$ is equigenerated square-free monomial ideal of degree 2, we have that $dim (\delta _{\mathcal{F}}(I))=1$. Also, from \cite[Corollary  6.3.5]{v},  we know that $dim (\delta _{\mathcal{N}}(I)))=n-ht(I)-1$. Now the equality of dimensions of these two complexes yields that $ht(I)=n-2$. Moreover, as $|G(I)|=r$,
$f_{1}(\delta _{\mathcal{F}}(I))=r$. Therefore, by \cite[Lemma 3.2]{deg2}, we have $f_{1}(\delta _{\mathcal{N}}(I))= {n\choose 2}-r$. The fact that $I$ is a quasi $f$-ideal of the type $(0,b)$ also gives us the equation: $f_{1}(\delta _{\mathcal{N}}(I))-f_{1}(\delta _{\mathcal{F}}(I))=b$, hence, ${n\choose 2}-r-r=b$  or  ${n\choose 2}-2r=b$. This further implies that $|G(I)| =r=\frac{1}{2}({n\choose 2}-b)$. Note that the equation ${n\choose 2}-2r=b$ also tells that the parity of ${n\choose 2}$ is same as the parity of $b$, and hence the condition $(ii)$.\\
\indent Conversely, we suppose that the conditions $(i)$, $(ii)$ and $(iii)$ are satisfied. As $I$ is equigenerated square-free monomial ideal of degree 2, $dim (\delta _{\mathcal{F}}(I)))=1$. The condition $(i)$ together with the fact that
 $dim (\delta _{\mathcal{N}}(I)))=n-ht(I)-1$ implies that the complex $\delta _{\mathcal{N}}(I)$ also has dimension 1. As both the complexes, $\delta _{\mathcal{F}}(I)$ and $\delta _{\mathcal{N}}(I)$, have the same vertex set $V=\{v_1,v_2,\ldots,v_n\}$, so $f_{0}(\delta _{\mathcal{F}}(I)= f_{0}(\delta _{\mathcal{N}}(I)= n$. This shows that $f_{0}(\delta _{\mathcal{N}}(I))-  (f_{0}(\delta _{\mathcal{F}}(I))= 0$. The other two conditions together with \cite[Lemma 3.2]{deg2} give the following: $$f_{1}(\delta _{\mathcal{N}}(I))-  (f_{1}(\delta _{\mathcal{F}}(I))= {n\choose 2}-|G(I)|-|G(I)|={n\choose 2}-2|G(I)|$$

 $$= {n\choose 2}-2 \frac{1}{2}({n\choose 2}-b)=b$$
 Thus $I$ is quasi $f$-ideal of type $(0,b)$.
\end{proof}

\begin{Remark}{\em
While characterizing $f$-ideals of degree $2$, we were bound to deal with only those polynomial rings in $n$ variables for which $n\choose 2$ was even. But this restriction is no longer needed for the case of quasi $f$-ideals of degree $2$, as the above theorem also considers the situation when $n\choose 2$ is odd. However, it  imposes some restrictions on the value of $b$; it says that if $I$ is quasi $f$-ideal of degree $2$ in the polynomial ring $R=k[x_{1},x_{2},...,x_{n}]$ having type $(0,b)$, then the parity of $n\choose 2$ and $b$ must be same. So, if $n=4k$ or $n=4k+1$ and $b$ is odd, then there will not be any quasi $f$-ideal of degree $2$ of type $(0,b)$. Similarly for $n=4k+2$ or $n=4k+3$, there will not be any quasi $f$-ideal of type $(0,b)$ with even $b$. Moreover, since $b=f_{1}(\delta _{\mathcal{N}}(I))-f_{1}(\delta _{\mathcal{F}}(I))$, \cite[Lemma 3.2]{deg2} shows that $|b|$ can not be greater than $n\choose 2$. Also, as the height of a quasi $f$-ideal of degree $2$ has to be $n-2$, $|b|\neq {n\choose 2}$.
 }
\end{Remark}

 The next theorem gives a combinatorial characterization of quasi $f$-ideals of degree $2$. This theorem involves the notion of an upper perfect set.

\begin{Theorem}
{\em Let $I$ be an equigenerated square-free monomial ideal of the polynomial ring  $R=k[x_{1},x_{2},...,x_{n}]$ of degree $2$, and let  $G(I)$ be its minimal generating set. Then $I$ is quasi $f$-ideal of type $(0,b)\in \mathbb{Z}^2$ (where $|b|<{n\choose 2}$) if and only if the following conditions are satisfied:
\begin{enumerate}
    \item the parity of ${n\choose 2}$ is same as the parity of $b$,
    \item the set $G(I)$ is upper perfect with $|G(I)|=\frac{1}{2}({n\choose 2}-b)$.
\end{enumerate}
}
\end{Theorem}

\begin{proof}
Let us first suppose that $I$ is quasi $f$-ideal of type $(0,b)$. The condition $(1)$ is the condition $(2)$ of Theorem 4.1; clearly, $|G(I)|=\frac{1}{2}({n\choose 2}-b)$, as $I$ is a quasi $f$-ideal of type $(0,b)$. Now we only have to show that $G(I)$ is upper perfect. This is equivalent to say that $I$ contains all square-free monomial of degree $3$. Indeed it is so, otherwise if there is some monomial (say) $x_{i_1}x_{i_2}x_{i_3}$ which does not belong to $I$, then the corresponding subset $\{v_{i_1},v_{i_2},v_{i_3}\} \in \delta _{\mathcal{N}}(I) $. This means that $dim (\delta _{\mathcal{N}}(I))\geq 2$, which is a contradiction to the fact that $dim (\delta _{\mathcal{N}}(I))=1$.\\
\indent Conversely, note that the conditions $(2)$ and $(3)$ of Theorem 4.1 directly follows from the assumption. We only have to show that $ht(I)=n-2$. As the ideal $I$ contains all square-free monomials of degree 3 and higher, so $dim (\delta _{\mathcal{N}}(I))\leq 1$. The fact that $|b|<{n\choose 2}$ yields that
$\delta _{\mathcal{N}}(I))$ is $1$-dimensional complex, which implies that $ht(I)=n-2$. Thus $I$ is quasi $f$-ideal of type $(0,b)$.
\end{proof}

It will be interesting to determine the bounds on the values of $b$, which is given in the proposition below.

\begin{Proposition}
{\em Let $I$ be a quasi $f$-ideal of degree $2$ and type $(0,b)$ in the polynomial ring $R=k[x_{1},x_{2},...,x_{n}]$. Then the following holds true:
$$ -{n\choose 2}+2\leq b\leq {n\choose 2}-2N(n,2) $$}
\end{Proposition}

\begin{proof}
As $I$ is quasi $f$-ideal of degree $2$ and type $(0,b)$, the
theorem 4.3 tells us that $G(I)$ is a perfect subset of $Sm(R)_2$ with
$|G(I)|=\frac{1}{2}({n\choose 2}-b)$. Since $N(n,2)$ is the smallest
cardinality of perfect sets of degree $2$, this means that $|G(I)|\geq N(n,2)$.
$$\Rightarrow  \frac{1}{2}({n\choose 2}-b) \geq N(n,2) $$

\begin{equation}\label{1}
    \Rightarrow  b \leq {n\choose 2}-2N(n,2)
\end{equation}
Moreover, it is clear that $|G(I)|\leq {n\choose 2}$. This means that $\frac{1}{2}({n\choose 2}-b)\leq {n\choose 2}$, which gives the  inequality:
$-{n\choose 2}\leq b$. However, if $b=-{n\choose 2}$, then $|G(I)|={n\choose 2}$. The square-free monomial ideal of degree $2$ with ${n\choose 2}$ generators has height $n-1$. Since the quasi $f$-ideal of degree $2$ and type $(0,b)$ should be of height $n-2$, so the value $b=-{n\choose 2}$ is not acceptable. Also, as the parity of $b$ and $-{n\choose 2}$ has to be same, the immediate acceptable value of $b$ which is greater than $-{n\choose 2}$ would be $-{n\choose 2}+2$.  The value $b=-{n\choose 2}+2$ can be realized for by any square-free monomial ideal of degree $2$ whose generating set consists of ${n\choose 2}-1$ generators. Thus we have that  $ -{n\choose 2}+2\leq b\leq {n\choose 2}-2N(n,2) $.
\end{proof}

Now let us talk about the associated prime ideals of these ideals. Consider a quasi $f$-ideal $I$ of degree $2$ and type $(0,b)$ and let $\mathfrak{p}$ be any minimal prime ideal of $I$. By the condition $(i)$ of Theorem 4.1, we have that $ht(\mathfrak{p})\in \{n-2,n-1\}$. In the next theorem, we see that which monomial prime ideals of height $n-2$ and $n-1$ belong to $Ass(R/I)$.

\begin{Theorem}{\em Let $I$ be a quasi $f$-ideal of degree $2$ and type $(0,b)$. Then the following statements are true:\\
$(i)$ A monomial prime ideal $\mathfrak{p}$ of height $n-2$ belongs to $Ass(R/I)$ if and only if the square-free quadratic monomial $x_ix_j \notin G(I)$, where $x_i,x_j\notin \mathfrak{p} $.\\
$(ii)$ A monomial prime ideal $\mathfrak{p}$ of height $n-1$ belongs to $Ass(R/I)$ if and only if $ x_ix_j \in G(I)$ for all $j$, where $x_i\notin \mathfrak{p}$.

}
\end{Theorem}

\begin{proof}

{We will use the well-known one-to-one correspondence between facets of $\delta _{\mathcal{N}}(I)$ and the minimal vertex covers of $\delta _{\mathcal{F}}(I)$, which correspond to the minimal primes of $I$ (as given in \cite{faridi}). More precisely, $F$ is a facet of $\delta _{\mathcal{N}}(I)$ if and only if $\mathfrak{p}=(x_i \ | \ i\notin F)$ is a minimal prime of $I$.\\
Case $(i)$: Since $I$ is square-free monomial ideal, the associated primes of $I$ are precisely the minimal primes. Thus, a monomial prime ideal $\mathfrak{p}$ of height $n-2$ belongs to $Ass(R/I)$ if and only if $\{v_i,v_j\}$ is a facet of $\delta _{\mathcal{N}}(I)$, where $x_i,x_j \notin \mathfrak{p}$. This is equivalent to say that $x_ix_j \notin G(I)$, because $I$ is quasi $f$-ideal of type $(0,b)$ and degree $2$.  \\
 Case $(ii)$: Let $\mathfrak{p}$ be a monomial prime ideal of height $n-1$ such that $x_i\notin \mathfrak{p}$. Then $\mathfrak{p}$ belongs to $Ass(R/I)$ if and only if $\{v_i\}$ is a facet (isolated vertex) of $\delta _{\mathcal{N}}(I)$.  This is equivalent to say that  the sets $\{v_i,v_j\}$ do not belong to $\delta _{\mathcal{N}}(I)$  for all $j\neq i$. Equivalently,  $x_ix_j \in G(I)$ for all $j\neq i$, because $G(I)\subset Sm(R)_2$.}
\end{proof}
We now move ahead and describe how quasi $f$-ideals of degree $2$ can be constructed. In particular, we show that every ordered pair $(0,b)$, where $ -{n\choose 2}+2\leq b\leq {n\choose 2}-2N(n,2)$, can be realized as type of some quasi $f$-ideal of degree $2$.

\begin{Proposition} Let $A$ be a proper subset of $\{x_1,x_2,\ldots,x_n\}$. Then for the subset $P_A=\{x_ix_j |x_i,x_j\in A \ or\ x_i,x_j\notin A\}$ of $Sm(R)_2$, the ideal $I=\langle P_A\rangle$ is a quasi $f$-ideal in $R=k[x_1,x_2,\ldots,x_n]$ of degree $2$ and type $(0, \frac {n-(n-2|A|)^{2}}{2})$.
\end{Proposition}

\begin{proof}
It is easy to see that $G(I)=W_{A}$ is a perfect set of $Sm(R)_2$. Moreover, as
$|W_{A}|={|A|\choose 2}+{{n-|A|}\choose 2} =\frac
{2|A|^2-2n|A|-n+n^2}{2}=\frac {1}{2}({n\choose 2}-\frac
{n-(n-2|A|)^{2}}{2})$, the theorem $4.3$ shows that $I$ is a quasi $f$-ideal
 of the type $(0, \frac {n-(n-2|A|)^{2}}{2})$.
\end{proof}

\begin{Proposition}{ \em Let $b$ be an integer satisfying the inequality of Proposition 4.4. Then the pair $(0,b)$ can be realized as the type of some quasi $f$-ideal provided $b$ and $n \choose 2$ have same parity.}
\end{Proposition}

\begin{proof} If $n \choose 2$ is even, then either $n=4t$ or $n=4t+1$. Similarly, if $n \choose 2$ is odd, then either $n=4t+2$ or $n=4t+3$. We consider each situation case by case.\\
Case $(i)$: $\underline{n=4t}$. For the subset $A=\{x_1,x_2,\ldots, x_{2t}\}$, we claim that the ideal $I=<W_A \cup D>$ is quasi $f$-ideal of type $(0,b)$, where $D$ is any subset of $Sm(R)_2 - W_A$ such that $|D|={\frac {1}{2}}({{n\choose 2}-b}) - |W_A|$. It is enough to show that
 ${\frac {1}{2}}({{n\choose 2}-b}) \geq |W_A|$, because only then we can form such $D$. Note that  $|W_{A}|= 2{{2t}\choose 2}=4t^2-2t$. The expression ${\frac {1}{2}}({{n\choose 2}-b})$ assumes the least value when $b$ takes the largest value which is $b={n\choose 2}-2N(n,2)$. Now as $n=4t$ is even, the lemma 3.3 of \cite{tswuPubl} tells that $b={4t\choose 2}-2 (4t^{2}-2t)=2t$. This means that ${\frac {1}{2}}({{n\choose 2}-b})={\frac {1}{2}}({{4t\choose 2}-2t})= 4t^{2}-2t=|W_{A}|$. Consequently, if $b<{n\choose 2}-2N(n,2)$, then we have strict inequality ${\frac {1}{2}}({{n\choose 2}-b})>|W_A|.$    \\
 Case $(ii)$: $\underline{n=4t+1}$. Consider the same subset $A=\{x_1,x_2,\ldots, x_{2t}\}$ and the same ideal $I=<W_A \cup D>$, where $D$ is any subset of $Sm(R)_2 - W_A$ such that $|D|={\frac {1}{2}}({{n\choose 2}-b}) - |W_A|$. In this case, $|W_{A}|= {{2t}\choose 2}+ {{2t+1}\choose 2}=4t^2$. Once again the least value of the expression ${n\choose 2}-2N(n,2)$ is assumed for the largest $b={n\choose 2}-2N(n,2)$. Now as $n=4t+1$ is odd, we have $b={n\choose 2}-2N(n,2)$ = ${{4t+1}\choose 2}-2 (4t^{2})=2t$. This means that ${\frac {1}{2}}({{n\choose 2}-b})={\frac {1}{2}}({{4t+1\choose 2}-2t})= 4t^{2}=|W_{A}|$. Finally, we have ${\frac {1}{2}}({{n\choose 2}-b})>|W_A|$ for each $b$ which is strictly less then ${n\choose 2}-2N(n,2)$. Similarly, taking $A=\{x_1,x_2,\ldots, x_{2t+1}\}$, it can be shown that ${\frac {1}{2}}({{n\choose 2}-b}) \geq |W_A|$ for the cases $n=4t+2$ and $n=4t+3$.
Now choosing any subset $D$ of $Sm(R)_2 - W_A$ such that $|D|={\frac {1}{2}}({{n\choose 2}-b}) - |W_A|$ and setting $G(I)={W_{A}}\cup D$ gives us quasi $f$-ideal $I$ of type $(0,b)$.
\end{proof}

\begin{Example}{\em
For $n=8$, we have $-26\leq b \leq 4$. In order to construct quasi $f$-ideal $I$ of type (say) $(0,-6)$, set $A=\{x_{1},x_{2},x_{3},x_{4}\}$. Then
$W_{A}=\{x_{1}x_{2},x_{1}x_{3},x_{1}x_{4},x_{2}x_{3},x_{2}x_{4},\\x_{3}x_{4},x_{5}x_{6},x_{5}x_{7},x_{5}x_{8},x_{6}x_{7},x_{6}x_{8},x_{7}x_{8}\}.$
Note that $|W_{A}|=12$. Also, as $|G(I)|$ should be equal to  ${\frac
{1}{2}}({8\choose 2}-(-6))=17$, so we can choose any subset $D$ of $Sm(R)_2 - W_A$ consisting of five elements. Let it be
$D=\{x_{1}x_{6},x_{2}x_{7},x_{2}x_{8},x_{3}x_{7}, x_{4}x_{7}\}.$ Then
$I=\langle {W_{A}}\cup D \rangle$ is a quasi $f$-ideal of degree $2$ and type
$(0,-6)$, because $G(I)$ is perfect and $|G(I)|=17={\frac
{1}{2}}({8\choose 2}-(-6)).$ Note that for any even $b$ satisfying $-26\leq b \leq 4$, we can construct quasi $f$-ideal of type (0,b) following the same procedure. }
\end{Example}

\begin{Theorem}
Let $I$ be a quasi $f$-ideal in $R=k[x_{1},x_{2},...,x_{n}]$
of the type $(0,b).$ Then the Hilbert function of $R/I$ is given by
$$H(S/I,z)=\frac {1}{4}\{(n^2-n+2b)z-n^2+5n-2b\}$$
\end{Theorem}

\begin{proof}
Suppose that $I$ is a quasi $f$-ideal of the type $(0,b)$, hence, $f_{1}(\delta _{\mathcal{N}}(I))=|G(I)|+b$. Therefore,
$f(\delta _{\mathcal{N}}(I))=(n,|G(I)|+b)$. By \cite[Proposition 6.7.3]{v}, the Hilbert function of
$R/I$ is given by $$H(K[\Delta],z)=\sum_{i=0}^{d}{{z-1}\choose
i}f_{i}={{z-1}\choose 0}f_{0}+{{z-1}\choose
1}f_{1}=n+(z-1)(|G(I)|+b)$$ On simplification, we get
$$H(S/I,z)=n+(z-1)(\frac {n(n-1)}{4}+{\frac {1}{2}}b)=\frac
{1}{4}\{(n^2-n+2b)z-n^2+5n-2b\}$$
\end{proof}

\begin{Theorem}
Let $I$ be a quasi $f$-ideal in $R=k[x_{1},x_{2},...,x_{n}]$
of the type $(0,b).$ Then the Hilbert series of $R/I$ is given by
$$F(R/I,z)=\frac {4+4(n-2)z+(n^2-5n+4+2b)z^{2}}{4(1-z)^{2}}$$
\end{Theorem}

\begin{proof}
Suppose that $I$ is a quasi $f$-ideal of the type $(0,b)$, the
definition implies that $f_{1}(\delta _{\mathcal{N}}(I))- f_{1}(\delta
_{\mathcal{F}}(I))=b$, which further means that $f_{1}(\delta
_{\mathcal{N}}(I))=|G(I)|+b$. Therefore, $f(\delta
_{\mathcal{N}}(I))=(n,|G(I)|+b)$. Now using \cite[Theorem 6.7.2]{v}, we have that
$$F(R/I,z)=\frac {f_{-1}z^0}{(1-z)^0}+\frac
{f_{0}z^1}{(1-z)^{1}}+\frac {f_{1}z^2}{(1-z)^2}$$
$$=1+\frac
{nz}{(1-z)}+\frac {(|G(I)|+b)z^2}{(1-z)^2}$$
and after simplification, we
get $$F(R/I,z)=\frac {1+(n-2)z+(1-n+|G(I)|+b)z^{2}}{(1-z)^{2}}$$
which implies the desired formula.
\end{proof}

\end{document}